\documentclass[11pt]{article} 

\usepackage{geometry} 
\usepackage{graphicx} 

\usepackage{booktabs} 
\usepackage{array} 
\usepackage{paralist} 
\usepackage{verbatim} 
\usepackage{subfig} 

\usepackage{fancyhdr} 
\usepackage{sectsty}
\allsectionsfont{\sffamily\mdseries\upshape} 

\usepackage[nottoc,notlof,notlot]{tocbibind} 
\usepackage[titles,subfigure]{tocloft} 

\usepackage{url,amsthm,amssymb,amsmath,amsfonts}
\newtheorem{theorem}{Theorem}[section]
\newtheorem{lemma}{Lemma}
\newtheorem{problem}{Problem}
\newtheorem{proposition}{Proposition}
\newtheorem{remark}{Remark}

\def\RR{\mathbb{R}}
\def\NN{\mathbb{N}}
\def\SS{\mathbb{S}}
\def\D{\mathrm{d}}

\title{Schoenberg coefficients and curvature at the origin of continuous isotropic positive definite kernels on spheres}
\author{Ahmed Arafat\footnote{\texttt{ah.arafat@mans.edu.eg}, Department of Mathematics, Mansoura University, Mansoura, Egypt} 
\and 
Pablo Gregori\footnote{\texttt{gregori@uji.es}, Instituto Universitario de Matem\'{a}ticas y Aplicaciones de Castell\'{o}n, Departamento de Matem\'{a}ticas, Universitat Jaume I de Castell\'{o}n, Campus Riu Sec, E-12071, Castell\'{o}n, Spain} 
\and 
Emilio Porcu\footnote{\texttt{emilio.porcu@newcastle.ac.uk}, Chair of Spatial Analytics Methods, School of Mathematics, Statistics and Physics, Newcastle University, UK} }

\begin{document}
\maketitle

\begin{abstract}
We consider the 
class $\Psi_d$ 
of continuous functions $\psi \colon [0,\pi] \to \mathbb{R}$, with $\psi(0)=1$ such that the associated isotropic 
kernel
$C(\xi,\eta)= \psi(\theta(\xi,\eta))$ ---with $\xi,\eta \in \mathbb{S}^d$ and $\theta$ the geodesic distance--- is positive definite on the product of two $d$-dimensional spheres $\mathbb{S}^d$. 
We face Problems 1 and 3 proposed in the essay \cite{gneiting_2013b}.
We have considered an extension that encompasses the solution of Problem 1 solved in \cite{fiedler_2013}, regarding the expression of the $d$-Schoenberg coefficients of members of $\Psi_d$ as combinations of $1$-Schoenberg coefficients. We also give expressions for the computation of Schoenberg coefficients of the exponential and Askey families for all even dimensions through recurrence formula. Problem 3 regards the curvature at the origin of members of $\Psi_d$ of local support. We have improved the current bounds for determining this curvature, which is of applied interest at least for $d=2$. 

\end{abstract}


\section{Introduction}\label{S:1}

There has been a fervent research activity around positive definite functions on spheres in the last five years \cite{Barbosa2015,
Beatson201722,
Castro201293,
estrade_2016,
fiedler_2013,
gneiting_2013a,
Guella2016286,
Guella201891,
Guella2016671,
Guella2016,
Guella2018150,
Peron_2016,
Menegatto2014189,
porcu_2016,
Trübner20173017,
Xu20182039,
ziegel_2014}. 
In particular, \cite{gneiting_2013a} offers an impressive overview on the problem as well as a number of connections between mathematical, complex and harmonic analysis, as well as approximation theory, with the theory of stochastic processes, Gaussian random fields, and geostatistics. 

Schoenberg's theorem \cite[Thm.~2]{schoenberg_1942}, in concert with the orthonormality properties of spherical harmonics, imply that a very natural assumption on positive definite functions over $d$-dimensional spheres of $\RR^{d+1}$ is that they depend on the geodesic (great circle) distance between any two points located over the $d$-dimensional spherical shell. Such an assumption is known as \textit{geodesic isotropy} and it is the building block for more sophisticated constructions, such as in \cite{berg_2016,estrade_2016} and \cite{porcu_2016}. More technical approaches based on complex spheres and locally compact groups have been proposed in \cite{berg_2016b}.

\cite{gneiting_2013b} culminates in a collection of open problems that have inspired mathematicians and statisticians, and we cite the works \cite{fiedler_2013,berg_2016,Peron_2016,ziegel_2014} and the \textit{tour de force} in \cite{beatson_2013}.

This paper faces two important problems, the former being related to the representation of the $d$-Schoenberg's coefficients (see Section \ref{sec2} below) in terms of $1$-Schoenberg coefficients. Such a problem is parenthetical to the celebrated Matheron's turning bands operator \cite{matheron_1963} proposed in Euclidean space only. In particular, a representation of the $d$-Schoenberg coefficients in terms of $1$- Schoenberg's coefficients (see subsequent sections for details) was provided by \cite{fiedler_2013} when $d$ is odd, and in terms of $2$-Schoenberg's coefficients when $d$ is even. The case of even dimension $d$ and a representation in terms of $1$-Schoenberg's coefficients is still elusive and constitutes one of the challenges of the present paper.

The latter problem finds instead motivation in atmospheric data assimilation, where locally supported isotropic correlation functions are used for the distance-dependent reduction of global scale covariance estimates in ensemble Kalman filter settings \cite{buehner_2007, hamill_2001}.

We culminate our findings by proposing closed forms of the $2$-Schoenberg's coefficients related to celebrated families of positive definite functions on spheres. One of them means an improvement, over the interesting expression found in \cite[p.729]{huang_zhang_robeson_2011}, with respect to the numerical computation, because we turn an infinite series into a finite sum.

The plan of the paper is the following. Section \ref{sec2} provides the necessary concepts, notation and theoretical tools. Section \ref{sec3} introduces the statements of problems 1 and 3 of \cite{gneiting_2013b} and follows with our improvements to their current solutions. Section \ref{sec:coefficients-special-families} includes closed form expressions for the $2$-Schoenberg coefficients of correlation functions in the exponential and Askey families. 

\section{The class $\Psi_d$ and $d$-Schoenberg coefficients}
\label{sec2}

This section is largely expository and details the necessary material needed for a self contained exposition. Let $d$ be a positive integer. We consider the $d$-dimensional sphere $\SS^{d}$ with unit radius, embedded in $\RR^{d+1}$ so that $\SS^d= \{x \in \RR^{d+1} : \|x\|=1 \}$. We define the \textit{geodesic} or \textit{great circle} distance as the mapping $\theta \colon \SS^d \times \SS^d \to [0,\pi]$ defined through $\theta(\xi,\eta)= \arccos(\langle \xi,\eta \rangle)$, with $\langle \cdot,\cdot \rangle$ denoting the classical dot product. Throughout, we shall be sloppy whenever using the abuse of notation $\theta$ for $\theta(\xi,\eta)$. We also consider the Hilbert sphere $\SS^{\infty}= \{x \in \RR^{\mathbb{N}}: \|x\|=1 \}$. We say that the function $C \colon \mathbb{S}^d\times\mathbb{S}^d\rightarrow\mathbb{R}$ is positive definite if 
\begin{equation*}
\sum_{i,j=1}^n\alpha_i\alpha_j C(\mathbf{x}_i,\mathbf{x}_j)\geq0,
\end{equation*}
for any $\alpha_1, \dots, \alpha_n\in\mathbb{R}$ and for every $\mathbf{x}_1, \dots, \mathbf{x}_n\in\mathbb{S}^d$.

We denote by $C_n^{\lambda}$ the $n$-th Gegenbauer polynomial of order $\lambda>0$, uniquely identified through the intrinsic relation 

\[
\frac{1}{\left ( 1+r^2-2r \cos \theta \right )^{\lambda}}= \sum_{n=0}^{\infty} r^n C_{n}^{\lambda} (\cos \theta), \qquad \theta \in [0,\pi],
\] 
where $r \in (-1,1)$. It is of fundamental importance that \cite[Eq.~18.14.4]{NIST:DLMF}
\begin{equation} \label{value-at-1}
| C_{n}^{\lambda} (x) | \le \frac{\Gamma(n+2\lambda)}{n! \Gamma(2 \lambda)}=C_{n}^{\lambda}(1) , \qquad x \in [-1,1].
\end{equation}



The trigonometric expansion in the following lemma is crucial for our first solution. We recall the notation of the \emph{rising factorial} $(x)_m := x (x+1) \cdots (x+m-1)$ for any real number $x$ and any non negative integer length $m$, with the convention $(x)_0=1$.

\begin{lemma}\label{lemma21}
Let $n \geq 1$ be an integer, $\lambda > 0$ and $0 < \theta < \pi$. Then the expansion
	\begin{equation}\label{24}
	(\sin\theta)^{2\lambda-1} C_n^{\lambda}(\cos\theta)= 
	\frac{2^{2-2\lambda}\Gamma(n+2\lambda)}{\Gamma(\lambda)\Gamma(n+\lambda+1) }
	\sum_{\mu=0}^{\infty} \frac{(1-\lambda)_\mu (n+1)_\mu}{\mu! (n+\lambda+1)_\mu} \sin(n+2\mu+1)\theta.
	\end{equation}
holds, and reduces to a finite sum (up to $\mu=\lambda -1$) whenever $\lambda$ is integer.
\end{lemma}

\begin{proof}
\cite[p.~93, Eq.~4.9.22]{szego_1939} states the expansion for $\lambda>0$, $\lambda \neq 1, 2, 3,\dots$ and $0 < \theta < \pi$. For the remaining case, we prove it by induction on $\lambda \in \{ 1, 2,\ldots \}$. For simplicity, let us denote $\gamma_{\lambda, n} := \frac{2^{2-2\lambda} \Gamma( n+2\lambda )}{\Gamma( \lambda ) \Gamma( n+\lambda+1 )}$ and $\beta_{\lambda,n,\mu} := \frac{( 1-\lambda )_\mu ( n+1 )_\mu}{ \mu! ( n+\lambda+1 )_\mu}$, and note that $\gamma_{1,n} = 1$ for all $n$ and $\beta_{\lambda,n,0} =1$ for any $\lambda$ and $n$.

To begin with, \cite[18.5.2]{NIST:DLMF} shows that $C_n^1(\cos \theta) = \sin((n+1)\theta) / \sin \theta$, hence Eq.~\eqref{24} holds for $\lambda=1$ and all $n$. Now, for $\lambda = k \geq 1$ we shall use the recurrence relation adapted from \cite[Eq.~4.7.27]{szego_1939}
$$
2 \lambda \sin^2 \theta C_n^{\lambda+1} (\cos \theta) = (n+2\lambda+1) \cos \theta C_{n+1}^\lambda (\cos \theta) - (n+2) C_{n+2}^\lambda (\cos \theta),
$$
the induction hypothesis (Eq.~\ref{24}), and the product-to-sum trigonometric identities, in order to prove Eq.~\ref{24} for $\lambda = k+1$ and all $n$, as follows:

\begin{multline*}
(\sin \theta)^{2(k+1)-1} C_{n}^{k+1}( \cos \theta ) =
\frac{(\sin \theta)^{2k-1}}{2k} 2 k \sin^2 \theta C_n^{k+1}( \cos \theta ) \hfill \\
= \frac{n+2k+1}{2k} (\sin \theta)^{2k-1} C_{n+1}^k( \cos \theta ) \cos \theta -  \frac{n+2}{2k} (\sin \theta)^{2k-1} C_{n+2}^k(\cos \theta )
\end{multline*}

\begin{multline*}
= \frac{n+2k+1}{2k} \gamma_{k, n+1} \sum_{\mu=0}^{k-1} \beta_{k, n+1, \mu} \sin( n+2\mu+2) \theta \cos \theta \\
\hfill - \frac{n+2}{2k} \gamma_{k,n+2} \sum_{\mu=0}^{k-1} \beta_{k, n+2, \mu} \sin( n+2\mu+3) \theta
\end{multline*}

\begin{multline*}
= \frac{n+2k+1}{4k} \gamma_{k, n+1} \sum_{\mu=0}^{k-1} \beta_{k, n+1, \mu} [ \sin( n+2\mu+3) \theta + \sin( n+2\mu+1) \theta ] \\
\hfill - \frac{n+2}{2k} \gamma_{k,n+2} \sum_{\mu=1}^{k} \beta_{k, n+2, \mu-1} \sin( n+2\mu+1) \theta
\end{multline*}

\begin{multline*}
= \gamma_{k+1, n} \sum_{\mu=1}^{k} \beta_{k, n+1, \mu-1} \sin( n+2\mu+1) \theta
+ \gamma_{k+1, n} \sum_{\mu=0}^{k-1} \beta_{k, n+1, \mu} \sin( n+2\mu+1) \theta \\
\hfill - \frac{2(n+2)}{n+k+2} \gamma_{k+1,n} \sum_{\mu=1}^{k} \beta_{k, n+2, \mu-1} \sin( n+2\mu+1) \theta \\
= \gamma_{k+1,n} \left[ \beta_{k, n+1, 0} \sin( n+1) \theta \phantom{\sum_{\mu=1}^{k}} \right. \hfill \\
+ \sum_{\mu=1}^{k} [ \beta_{k, n+1, \mu-1} + \beta_{k, n+1, \mu} - \frac{2(n+2)}{n+k+2} \beta_{k, n+2, \mu-1} ]  \sin( n+2\mu+1) \theta \\
\hfill \left. \phantom{\sum_{\mu=1}^{k}} + [ \beta_{k, n+1, k-1} - \frac{2(n+2)}{n+k+2} \beta_{k, n+2, k-1} ]  \sin( n+2\mu+1) \theta \right]
\end{multline*}

\begin{multline*}
= \gamma_{k+1, n} \sum_{\mu=0}^{k} \beta_{k+1, n, \mu} \sin( n+2\mu+1) \theta \hfill
\end{multline*}
and the proof is complete, after the last step is thoroughly checked.

\end{proof}


Let  $\Psi_d$ be the class of continuous mappings $\psi \colon [0,\pi] \to \RR$ with $\psi(0)=1$ such that the continuous functions $C \colon \SS^d \times \SS^d \to \RR$ defined through $C(\xi,\eta)= \psi(\theta(\xi,\eta))$ are positive definite. The dimension $d$ and the parameter $\lambda$ are related by $\lambda := \frac{d-1}{2}$, and in the sequel, and for ease of notation, we shall use one or the other interchangeably. \cite{schoenberg_1942} characterized the positive definite functions defined on the spheres of any dimension.

\begin{theorem}\cite{schoenberg_1942}\label{Th:1}
	A necessary and sufficient condition for a continuous mapping $\psi \colon [0,\pi] \to \RR $, with $\psi(0)=1$ to belong to the class $\Psi_d$ is that the ultraspherical expansion
	\begin{equation}\label{eq:schoenberg-expansion}
	\sum_{n=0}^\infty\left\{ \frac{(n+\lambda) \Gamma(\lambda)}{\Gamma(\lambda+\frac{1}{2}) \Gamma(\frac{1}{2})} \frac{(n+1) \Gamma(2\lambda)}{\Gamma(n + 2 \lambda)} \cdot\int_0^\pi C_n^\lambda( \cos \theta' ) \psi(\theta') \sin^{2\lambda} \theta' \D \theta' \right\} C_n^\lambda( \cos \theta)
	\end{equation}
	has non-negative coefficients and converges absolutely and uniformly to $\psi(\theta)$ throughout $0\leq\theta\leq\pi$.
\end{theorem} 

\cite{gneiting_2013a} used Theorem \ref{Th:1} to characterize the members of class $\Psi_d$ through the representation
\begin{equation} \label{eq:gneiting-expansion}
\psi(\theta)= \sum_{n=0}^{\infty} b_{n,d} \frac{C_{n}^{\lambda}(\cos \theta)}{C_{n}^{\lambda}(1)}, \qquad \theta \in [0,\pi],
\end{equation}
with $\{ b_{n,d} \}_{n=0}^\infty $ being a uniquely identified probability mass system. We follow \cite{daley_2014} and \cite{ziegel_2014} when referring to $b_{n,d}$ as \emph{$d$-Schoenberg coefficients}. 

The classes $\Psi_d$ are nested, with the inclusion relation 
$$ \Psi_1 \supset \Psi_2 \supset \cdots \supset \Psi_{\infty}:= \bigcap_{d \ge 1} \Psi_d $$
being strict, and where $\Psi_{\infty}$ has as direct relation to the Hilbert sphere as previously defined. 

\cite{gneiting_2013a} and \cite{beatson_2013} obtain recurrent formulas in order to write coefficient $b_{n,d}$ as a linear combination of $b_{n, d-2}$ and $b_{n+2, d-2}$. By applying recursivity, each coefficient $b_{n,d}$ can be finally written, when $d$ is odd, as a linear combination of Schoenberg coefficients in the circle, $\{b_{n+2k,1}\}_{k=0}^{\lfloor d/2 \rfloor}$, and when $d$ is even, as a linear combination of Schoenberg coefficients in the sphere, $\{b_{n+2k,2}\}_{k=0}^{d/2}$.

By orthogonality of the Gegenbauer polynomials, we can identify coefficients of Eq. \eqref{eq:schoenberg-expansion} and \eqref{eq:gneiting-expansion} and get \cite[Cor.~2]{gneiting_2013a}
\begin{equation}\label{2.5}
b_{n,d}=\frac{(n+\lambda)\Gamma(\lambda)}{\Gamma(\lambda+1/2)\Gamma(1/2)}\int_0^\pi C_n^{\lambda}(\cos\theta)\psi(\theta)(\sin\theta)^{2\lambda}\D\theta,
\end{equation}
where as usual $\lambda:=(d-1)/2$.

We recall that $1$-Schoenberg coefficients of are the Fourier coefficients for even functions:
\begin{equation} \label{eq:coeff-real-line}
b_{0,1} := \frac{1}{\pi} \int_0^\pi \psi(\theta) \D \theta, \qquad b_{n,1} := \frac{2}{\pi} \int_0^\pi \psi(\theta) \cos (n\theta) \D \theta, \qquad (n \geq 1).
\end{equation}

\section{Gneiting's problems and current solutions}
\label{sec3}

\subsection{Statements of the problems}

We now expose the problems faced in the paper together with their partial solutions. 
\begin{problem}\cite[Problem 1]{gneiting_2013b}
	Let $n\ge0$ and $k \ge 1$ be integers. Find the coefficients $a_{n,1},\ldots,a_{n,k}$ in the expansion
	\begin{equation*}
	b_{n,2k+1}= \sum_{i=0}^k a_{n,i} b_{n+2i,1}
	\end{equation*}
	associated to the $(2k+1)$-Schoenberg coefficients in terms of Fourier coefficients $b_{n,1}, \ldots$, $b_{n+2k,1}$. Similarly, find the $(2k+2)$-Schoenberg coefficients in terms of the $2$-Schoenberg coefficients $b_{n,2},b_{n+2,2}, \ldots,b_{n+2k,2}$.
\end{problem}

In order to state Problem $2$, we follow \cite{gneiting_2013a} when calling $\Psi_{d}^{c}$ the subclass of $\Psi_d$ having members $\psi$ that vanish for any $\theta \ge c$, with $c \in (0,\pi]$. When $c< \pi$, then any member of $\Psi_d^{c}$ is called \emph{locally supported}, otherwise it is called \emph{globally supported}.

\begin{problem}\cite[Problem 3]{gneiting_2013b}
	For an integer $d \ge 1$, and for a given $c \in (0,\pi]$, find 
	\begin{equation} \label{curvature}
	a_{d}^c: = \inf_{\psi \in \Psi_d^c} \left ( - \psi^{''}(0) \right ). 
	\end{equation} 
\end{problem}

This is a problem of applied interest when $d = 2$. In atmospheric data assimilation, locally supported isotropic correlation functions are used for the distance-dependent reduction of global scale covariance estimates in ensemble Kalman filter settings (see \cite{gneiting_2013a} and the references therein). Thus, it is appealing to use a member of the class $\Psi_2^c$ with minimal curvature at the origin.

Some comments are in order. The solution of Problem $1$ requires the use of recursive formulae for the Gegenbauer polynomials and a constructive argument that will be exposed subsequently. An approach of Problem $2$ relies on considering $\widetilde{\Psi}_d^c$, the subclass of $\Psi_d$ given by those members $\psi \in \Psi_d$ such that $\psi(c)= 0$. Clearly, we have
\begin{equation}\label{increasion}
\Psi_{d}^c \subset \widetilde{\Psi}_d^c \subset \Psi_d,
\end{equation}
with the inclusion relation being strict. The definition of the class $\widetilde{\Psi}_d^c$ in concert with Schoenberg's representation and the oscillatory nature of Gegenbauer polynomial implies that, for any member of the class $\Psi_d$, there exists a collection of members $\widetilde{\psi}$ of the class $\Psi_d^{c_k}$, for $ \{ c_k\} $ being a sequence of constants with $c_k \in (0,\pi]$, such that 
$$ \psi(\theta) = \sum_{k=0}^{\infty} b_{k,d} \widetilde{\psi}_k(\theta), \qquad \theta \in [0,\pi]. $$

Another relevant comment is that Theorems 2 and 3 in \cite{gneiting_2013a} provide the upper bound $a_d^c \leq  \frac{1}{c^2} \frac{4}{d} j_{\frac{d-2}{2}}^2$,
where $j_\nu$ denotes the first positive zero of the Bessel function $J_\nu$. Some of these zeros are:
$$
j_0 \approx 2.4048, \quad j_{0.5}\approx\pi\quad j_1 \approx 3.8317.
$$
According to \cite{ehm_2004} the constant $a_d^c$ in Euclidean spaces depend on Boas-Kac roots, but \cite{ziegel_2014} showed that the convolution root does not always exist for positive definite functions on spheres. This makes the problem mathematically more interesting, and certainly tricky. 

\subsection{Main results}

\begin{proposition}
	Let $d > 1$ be an integer, and let $\lambda := (d-1)/2$. Then,
	\begin{multline} \label{fiedler}
	b_{n,d} = \frac{\sqrt{\pi} \Gamma(n + 2 \lambda)}{2^{2\lambda} \Gamma(\lambda+1/2) \Gamma(n+\lambda)}
	\left[
	b_{n,1}^* \phantom{\sum_{mu=1}^\infty} \right. \\ \left. - 
	\lambda 
	\sum_{\mu=1}^\infty 
	\frac{(1-\lambda)_{\mu-1} (n+1)_{\mu-1} (n+2\mu)}{\mu! (n+\lambda+1)_\mu}
	b_{n+2\mu, 1}
	\right]
	\end{multline}
	for $n \geq 0$, where $b_{n,1}^* = b_{n,1}$ when $n \geq 1$ and $b_{0,1}^* = 2 b_{0,1}$. If $d$ is odd, the expression involves only a finite number of coefficients, i.e., $b_{n,1}^*, b_{n+2,1}, \ldots, b_{n+2\lambda,1}$.
\end{proposition}

\begin{proof}
By plugging Eq.~\eqref{24} into Eq.~\eqref{2.5}, taking into account Eq.~\eqref{value-at-1}, and using again the product-to-sum trigonometric identities, we get 
	\begin{align}\label{26}
	b_{n,d} &= \alpha_{\lambda,n} \int_0^\pi \left[ \sum_{\mu=0}^\infty \beta_{\lambda,n,\mu} \sin(n+2\mu+1)\theta \sin\theta \right] \psi(\theta) \D\theta \nonumber \\
	&=\frac{\alpha_{\lambda,n}}{2} \int_0^\pi \left[ \beta_{\lambda,n,0} \cos(n \theta) + \sum_{\mu=1}^\infty [\beta_{\lambda,n,\mu} - \beta_{\lambda,n,\mu-1}] \cos(n+2\mu)\theta \right] \psi(\theta) \D\theta
	\end{align}
	where $\alpha_{\lambda,n} := \frac{2^{2-2\lambda}\Gamma(n+2\lambda)}{\Gamma(\lambda+1/2)\Gamma(1/2) \Gamma(n+\lambda)}$ and $\beta_{\lambda,n,\mu} :=\frac{(1-\lambda)_\mu (n+1)_\mu}{\mu! (n+\lambda+1)_\mu}$. 
	If $\lambda$ is integer (i.e. $d$ is odd), the series is a finite sum (up to index $\mu=\lambda-1$). Otherwise, we need to assume the uniform convergence of the series in $(0, \pi)$, in order to exchange the integral and the series signs. In that case, we use the definitions of $1$-Schoenberg coefficients \eqref{eq:coeff-real-line}: for $n>0$ we get
$$
b_{n,d} = \frac{\pi \alpha_{\lambda,n}}{4} \left[ b_{n,1} + \sum_{\mu=1}^\infty [\beta_{\lambda,n,\mu} - \beta_{\lambda,n,\mu-1}] b_{n+2\mu,1} \right],
$$
while for the special case $n = 0$ we have
$$
b_{0,d} = \frac{\pi \alpha_{\lambda,0}}{4} \left[ 2 b_{0,1} + \sum_{\mu=1}^\infty [\beta_{\lambda,0,\mu} - \beta_{\lambda,0,\mu-1}] b_{2\mu,1} \right],
$$
since $\beta_{\lambda,n,0} = 1$ for all $n$ and $\lambda$. For $\mu = 1,2,3,\ldots$ we can simplify
	\begin{equation*}
	\beta_{\lambda,n,\mu}-\beta_{\lambda,n,\mu-1} = \frac{-\lambda(n+2\mu)}{\mu(n+\lambda+\mu)} \beta_{\lambda,n,\mu-1},
	\end{equation*}
and we only need to prove that the aforementioned convergence of the series present at Eq. \eqref{26} is uniform in $0 < \theta < \pi$. The boundedness of the cosine functions allows to prove the uniform convergence from the absolute convergence of the coefficients $\sum_{\mu=1}^\infty |\beta_{\lambda,n,\mu} - \beta_{\lambda,n,\mu-1} | = \sum_{\mu=1}^\infty \frac{\lambda(n+2\mu)}{\mu(n+\lambda+\mu)} |\beta_{\lambda,n,\mu-1}|$.
	A closer look at the definition of $\beta_{\lambda,n,\mu-1}$ reveals that cancellation occurs for $\lambda$ non integer (but $2\lambda$ integer by definition) and $\mu$ large enough:
$$
\beta_{\lambda,n,\mu-1} = \frac{(1 - \lambda) (2 - \lambda) \cdots (n - \lambda)}{1 \cdot 2 \cdots n} \cdot \frac{\mu(\mu+1) \cdots (\mu+n-1)}{(\mu - \lambda) (\mu - \lambda+1) \cdots (\mu + n + \lambda - 1)}.
$$
It is indeed a quotient of polynomials in $\mu$ (for each fixed $\lambda$ and $n$) of respective degrees $n$ and $n + 2\lambda$. Considering the extra factor in the series, its general term is a quotient of polynomials of respective degrees $n+1$ and $n + 2\lambda +2$, which turns it to be convergent since $d > 1$ (i.e. $2 \lambda > 0$). Simple simplifications lead to the final expression. If $d$ is odd, the infinite series is a finite sum indeed, and we do not need to invoke the uniform convergence argument.
\end{proof}
\begin{remark}
In case that $\lambda\in\NN$ ---i.e., $d$ is odd---, Eq.~\eqref{fiedler} coincides, after simplification, with the one of Theorem~2.1 in \cite{fiedler_2013}. It is worth mentioning that \cite{fiedler_2013} used induction in order to prove the expression of Schoenberg coefficients for even (\emph{resp.} odd) dimension with respect to coefficients in the sphere (\emph{resp.} circle), in contrast to our direct derivation. 
\end{remark}

We are now able to face Problem 2, where a formal statement for a partial solution is exposed in the following. 
\begin{proposition}\label{prop2}
	Let $d>1$ be an integer. Then:
	\begin{itemize}
		\item[{\rm (i)}]
		$a_d^c \geq \frac{1}{1-\cos c}$ if $c \in [\pi/2, \pi]$.
		\item[{\rm (ii)}]
		$a_d^c \geq \frac{(d+1) (\cos c) (2 - \cos c) + 1}{(1-\cos c) ((d+1) \cos c +1)}$ if $c \in [\arccos \sqrt{\frac{1}{d+1}}, \pi/2]$.
	\end{itemize}
\end{proposition}
\begin{proof}
	It is easy to check \cite{beatson_2013} that $-\psi''(0) = \frac{1}{d} \sum_{n=1}^\infty n(n+d-1)b_{n,d}$ for any $\psi \in \Psi_d$ with associated $d$-Schoenberg coefficients $\{b_{n,d}\}_{n=0}^\infty$. Since the sequence $\{b_{n,d}\}_{n=0}^\infty$ is a probability mass system, functions $\psi$ with mass concentrated in lower index coefficients have a lower value of $-\psi''(0)$, and it is essential for the estimation of the infimum $a_d^c$ in Equation (\ref{curvature}).
	
	The set $\Psi_d^c$ is difficult to tackle, because locally supported functions have an infinite number of non null $d$-Schoenberg coefficients. In view of this, we consider $\Psi_d^c$ as a subset of the more amenable set $\widetilde{\Psi}_d^c := \{ \psi \in \Psi_d : \psi(c)=0 \}$, of functions having at least one zero at the fixed value $\theta = c$. Now, denote 
	\begin{equation*}
	\widetilde{a}_d^c := \inf_{\psi \in \widetilde{\Psi}_d^c} [- \psi''(0)].
	\end{equation*}
	Obviously, we have  $a_d^c \geq \widetilde{a}_d^c$ thanks to (\ref{increasion}), and the latter value is attainable at a known function for a range of values of $c$, as we shall show. In order to get $\widetilde{a}_d^c$ we need to solve the pair of equations
	\begin{equation} \label{eq:eq-infimum}
	\sum_{n=0}^\infty b_{n,d} = 1 \qquad \text{ and } \qquad \sum_{n=0}^\infty b_{n,d} \frac{C_n^\lambda(\cos c)}{C_n^\lambda(1)} = 0
	\end{equation}
	subject to the restriction $\{b_{n,d}\}_{n=0}^\infty \subset [0, \infty)$. As already stated, we shall check the values for functions with mass concentrated into the first coefficients. The constant function (i.e. $b_{n,d} = 0$ for $n \geq 1$) is clearly out of $\widetilde{\Psi}_d^c$. Thus, we check functions with $b_{n,d} = 0$ for $n \geq 2$. Using the equations in (\ref{eq:eq-infimum}) we get the single function
	$$
	\psi_c(\theta) = \frac{-\cos c}{1 - \cos c} + \frac{1}{1 - \cos c} \cos \theta,
	$$
	and a sufficient condition for $\psi$ to belong to the class $\psi \in \widetilde{\Psi}_d^c$ is that $c \in [\pi/2, \pi]$, with $-\psi_c''(0) = \frac{1}{1 - \cos c}$. Hence for $c \in [\pi/2, \pi]$ we have that $\psi_c \in \widetilde{\Psi}_d^c$ leading to $\widetilde{a}_d^c = \frac{1}{1 - \cos c}$.
	
	For $c \in [0, \pi/2]$ we have no members of $\widetilde{\Psi}_d^c$ with $b_{n,d} = 0$ for $n \geq 2$, and we shall look for functions with $b_{n,d} = 0$ for $n \geq 3$. Using again the system (\ref{eq:eq-infimum}) we get the set of functions that can be written as
	\begin{multline*}
	\psi_\beta(\theta) = - \frac{\cos c}{1 - \cos c} + \frac{(d+1) \cos c + 1}{d} \beta \\
	+ \left( \frac{1}{1 - \cos c} - \frac{(d+1)(1+\cos c)}{d} \beta \right) \cos \theta + \beta \frac{(d+1) \cos^2 \theta - 1}{d},
	\end{multline*}
	indexed by a parameter $\beta := b_{2,d}$. The non negativity restriction of their coefficients turns into the inequality
	\begin{equation}\label{Eq 3.8}
	\frac{d \cos c}{(1-\cos c)((d+1)\cos c + 1)} \leq \beta \leq \frac{d}{(d+1)\sin^2 c},
	\end{equation}
	which leads to a non empty set of values only if $c \geq \arccos \sqrt{ \frac{1}{d+1} }$, and $\widetilde{a}_d^c$ is attained for $\psi_\beta$ when $\beta$ attaches to the left-hand side of inequality (\ref{Eq 3.8}). This completes the proof. 
\end{proof}

This strategy might lead to values of $\widetilde{a}_d^c$ for a wider range of values $c$, by using functions with $b_{n,d} = 0$ for $n \geq 4$, and so on, but we have not explored further this line because of the complexity of equations. Another way (yet unexplored) of improving the lower bounds is using slightly more complex auxiliary sets $\Psi_d^{(c,c')}$ of functions having at least two zeros, or even more. We could find no examples of members of this subclass.

Figure \ref{fig:bounds-for-adc} depicts both upper and lower bounds for the range of $c$ in dimension $d=2$.

\begin{figure}[hbp]
	\centering
	\includegraphics[width=4in]{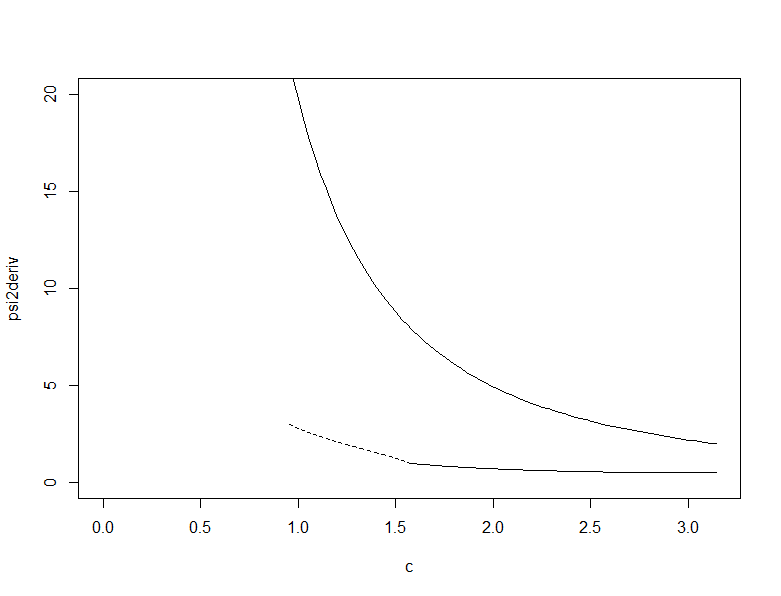}
	\caption{Upper (\cite[Theorem 5.1]{ehm_2004}) and lower (Proposition \ref{prop2}) bounds for $a_d^c$ in the range $c \in [ \arccos \sqrt{\frac{1}{d+1}}, \pi]$ for $d=2$.}
	\label{fig:bounds-for-adc}
\end{figure}

\section{On the $2$-Schoenberg coefficients of some celebrated parametric families}
\label{sec:coefficients-special-families}

This section inspects the problem of giving closed form expressions for the $2$-Schoenberg coefficients of correlation functions in the exponential  and Askey's families \cite{moller_2015}. 

A relevant remark is that what really matters is the computation of the $1$- and $2$-Schoenberg coefficients, because all the others can then be calculated inductively by using Corollary 3 in \cite{gneiting_2013a}. In particular, using Theorem 4.2 in \cite{moller_2015} one can even get the Schoenberg's coefficients related to the representation of a given member of the class $\Psi_{\infty}$. Since the $1$-Schoenberg coefficients for the exponential and Askey families have been provided in \cite{moller_2015}, we focus here on the tricky case of the $2$-Schoenberg coefficients related to these families. It is worth noting that \cite{huang_zhang_robeson_2011} analyzed the validity of several families of covariance functions over the sphere, and provided an explicit formula of the coefficients of the exponencial one. We derive another formula, more suitable for computation, since it is a sum of a finite numbers of terms, in contrast with the formula given in \cite[p.729]{huang_zhang_robeson_2011}.

First, we note that Gegenbauer polynomials simplify to Legendre polynomials $P_n$ when dealing with $\mathbb{S}^2$. Thus, classical Schoenberg's representation reduces to
\begin{equation*}
\psi(\theta)=\sum_{n=0}^\infty b_{n,2} P_n(\cos\theta),  \qquad \theta \in [0,\pi],
\end{equation*}
where
$$
b_{n,2} = \left(n+\frac{1}{2}\right) \int_0^\pi P_n(\cos\theta) \psi(\theta) \sin\theta\D \theta,
$$
for all $n\geq0$. The following representation for Legendre polynomials turns to be useful \cite{dixit2015finite}:
$$
P_n(\cos\theta)=2^n\sum_{m=0}^n \binom{n}{m}\binom{\frac{n+m-1}{2}}{n}(\cos\theta)^m.
$$
In view of the expression above, the $2$-Schoenberg coefficients can be computed through
\begin{equation}\label{Eq:4.2}
b_{n,2} = (2n+1)2^{n-1} \sum_{m=0}^n \binom{n}{m}\binom{\frac{n+m-1}{2}}{n}\int_0^\pi (\cos\theta)^m \sin\theta\psi(\theta)\D \theta.
\end{equation}

\subsection{Exponential Family}

Let us consider the exponential family (included in $\Psi_{\infty}$), given by
\begin{equation}\label{Eq:4.3}
\psi_{\alpha}(\theta)=\exp\left(-\frac{\theta}{\alpha}\right), \qquad \theta \in [0,\pi], 
\end{equation} 
with $\alpha$ being a positive scaling parameter. 
\begin{proposition}\label{lem:4.1} 
	The $2$-Schoenberg coefficients of functions $\psi_{\alpha}$ in Equation (\ref{Eq:4.3}) are given by 
	\begin{align}\label{Eq:4.4}
	\normalfont
	b_{n,2}(\alpha) = \frac{2n+1}{2^{1-n}}\ &\left\{ \sum_{m\equiv 0 (\normalfont\text{mod} 2)}^n \binom {n} {m}  \binom {\frac{n+m-1}{2}} {n} \frac{\left(1+\mathrm{e}^{-\frac{\pi}{\alpha}}\right)}{(m+1)2^m}\cdot\right.\nonumber\\
	&\qquad \left.\left[2^m-\sum_{k=0}^{\frac{m}{2}} \frac{1}{(2k+1)^2\alpha^2+1}\binom{m+1}{\frac{m-2k}{2}}\right]+\right.\nonumber\\
	&\qquad\left.\sum_{m\equiv 1 (\normalfont\text{mod} 2)}^n \binom {n} {m}  \binom {\frac{n+m-1}{2}} {n} \frac{\left(1-\mathrm{e}^{-\frac{\pi}{\alpha}}\right)}{(m+1)2^m}\cdot\right.\nonumber\\
	&\qquad\left.\left[2^m-\frac{1}{2}\binom {m+1}{\frac{m+1}{2}}-\sum_{k=1}^{\frac{m+1}{2}}\frac{1}{4k^2\alpha^2+1}\binom{m+1}{\frac{m-2k+1}{2}}\right]\right\}. 
	\end{align}
\end{proposition}

\begin{proof}
	We provide a proof by direct construction. We first use Equation (\ref{Eq:4.2}) to obtain 
	\begin{equation}\label{Eq:4.5}
	b_{n,2}(\alpha)=(2n+1)2^{n-1} \sum_{m=0}^n \binom{n}{m}\binom{\frac{n+m-1}{2}}{n}\int_0^\pi \mathrm{e}^{\left(-\frac{\theta}{\alpha}\right)}(\cos\theta)^m \sin\theta\D \theta.
	\end{equation}
	Using integration by parts, we have 
	\begin{equation}\label{Eq:4.66}
	\int_0^\pi \mathrm{e}^{\left(-\frac{\theta}{\alpha}\right)}(\cos\theta)^m \sin\theta\D \theta=\left(\mathrm{e}^{-\frac{\pi}{\alpha}}\ (-1)^m+1\right)-\frac{1}{\alpha(m+1)} \int_0^\pi \mathrm{e}^{-\frac{\theta}{\alpha}}(\cos\theta)^{m+1}\D\theta.
	\end{equation}
	To compute the second term of (\ref{Eq:4.66}), we use the explicit formulae proposed in \cite[Page 228]{jeffrey_2007} as follows: when $m$ is even, the integral on the right hand side of (\ref{Eq:4.66}) is given by 
$$
\int_0^\pi \mathrm{e}^{-\frac{\theta}{\alpha}}(\cos\theta)^{m+1}\D\theta=\frac{\left(1+\mathrm{e}^{-\frac{\pi}{\alpha}}\right)}{2^m}\sum_{k=0}^{\frac{m}{2}}\binom{m+1}{\frac{m-2k}{2}}\frac{\alpha}{(2k+1)^2\alpha^2+1}\label{Eq:4.77}
$$
while, for odd $m$, we obtain
\begin{align}
		\int_0^\pi \mathrm{e}^{-\frac{\theta}{\alpha}}(\cos\theta)^{m+1}\D\theta&=
		\binom{m+1}{\frac{m+1}{2}}\frac{\alpha\left(1-\mathrm{e}^{-\frac{\pi}{\alpha}}\right)}{2^{m+1}}\nonumber\\
		&\quad+ \frac{\left(1-\mathrm{e}^{-\frac{\pi}{\alpha}}\right)}{2^m} \sum_{k=1}^{\frac{m+1}{2}}\binom{m+1}{\frac{m-2k+1}{2}}\frac{\alpha}{(2k\alpha)^2+1}.	\label{Eq:4.88}	
\end{align}
We can now merge (\ref{Eq:4.77}) and (\ref{Eq:4.88}) into (\ref{Eq:4.66}) to obtain 
	\begin{align}\label{Eq:4.99}
	\int_0^\pi \mathrm{e}^{\left(-\frac{\theta}{\alpha}\right)}(\cos\theta)^m \sin\theta\D \theta&=\overbrace{\left(1+(-1)^m\mathrm{e}^{-\frac{\pi}{\alpha}}\ \right)}^{\text{general } m} - \overbrace{\frac{\left(1-\mathrm{e}^{-\frac{\pi}{\alpha}}\right)}{(m+1)2^{m+1}} \binom{m+1}{\frac{m+1}{2}}}^{m \text{ is odd}}\nonumber\\
	&-\overbrace{\frac{\left(1+\mathrm{e}^{-\frac{\pi}{\alpha}}\right)}{(m+1)2^m}\sum_{k=0}^{\frac{m}{2}}\binom{m+1}{\frac{m-2k}{2}}\frac{1}{(2k+1)^2\alpha^2+1}}^{m \text{ is even}}\nonumber\\
	&- \underbrace{\frac{\left(1-\mathrm{e}^{-\frac{\pi}{\alpha}}\right)}{2^m} \sum_{k=1}^{\frac{m+1}{2}}\binom{m+1}{\frac{m-2k+1}{2}}\frac{\alpha}{(2k\alpha)^2+1}}_{m \text{ is odd}}
	\end{align}
	Going back by substitution into Equation (\ref{Eq:4.99}) in (\ref{Eq:4.5}), we obtain (\ref{Eq:4.4}). This completes the proof.
\end{proof}

\begin{figure}[hbp]
	\begin{center}
		\begin{tabular}{c}
			\includegraphics[width=0.3\textwidth,height=35mm]{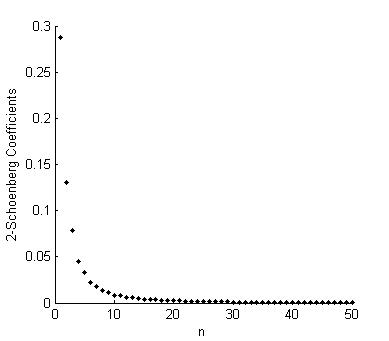}
			\includegraphics[width=0.3\textwidth,height=35mm]{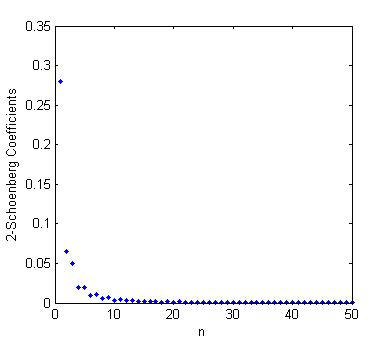}
			\includegraphics[width=0.3\textwidth,height=35mm]{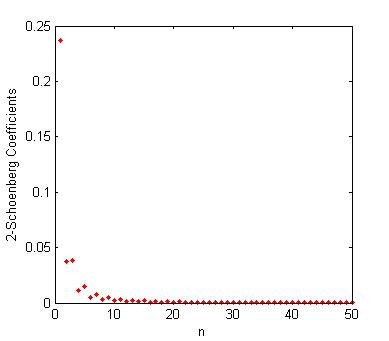}
			\\
		\end{tabular}
		\begin{tabular}{c}
			\includegraphics[width=0.3\textwidth,height=35mm]{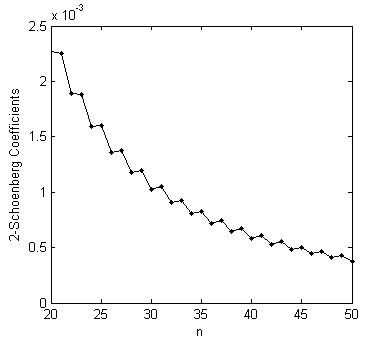}
			\includegraphics[width=0.3\textwidth,height=35mm]{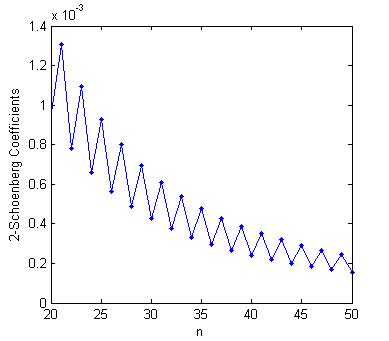}
			\includegraphics[width=0.3\textwidth,height=35mm]{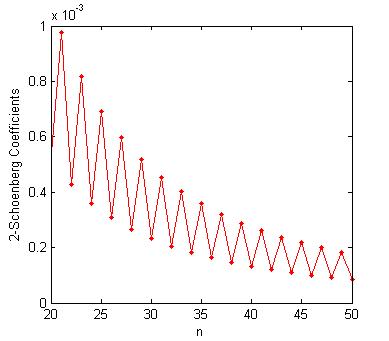}
			\\
		\end{tabular}
	\end{center}
	\caption{
		2-Schoenberg coefficients of the Exponential function $\psi_{\alpha}$ in Equation (\ref{Eq:4.3}), for $\alpha= 1,2,3$ (from left to right). Top: from 1st to 50-th coefficient; bottom: from 20-th to 50-th coefficient, showing the oscillating behavior away from the origin.}\label{fig:cc}
\end{figure}

\subsection{Askey Family}
The Askey function \cite{askey_1973} $\psi_{\alpha,\tau}$, is defined through 
\begin{equation}\label{Eq:4.10}
\psi_{\alpha,\tau}(\theta)=\left(1-\frac{\theta}{\alpha}\right)_+^\tau\quad \text{ for } \theta\in[0,\pi],
\end{equation}
where $\alpha > 0$ and $\tau\geq (d+1)/2$ are sufficient conditions for $\psi_{\alpha,\tau}$ to belong to the class $\Psi_d$. Since we are concerned with the $2$-Schoenberg's coefficients, we consider the case $\tau=2$. A relevant remark is that the Askey function is locally supported when $0<\alpha<\pi$.

\begin{proposition}\label{lem:4.2} 
	The $2$-Schoenberg coefficients related to the members $\psi_{\alpha,2}$ of the class $\Psi_{2}$ as in Equation (\ref{Eq:4.10}) are given by 
	\begin{multline}\label{Eq:4.11}
	b_{n,2}(\alpha) = (2n+1)\ 2^{n-1} \left\{ \sum_{m\equiv 0 (\text{mod} 2)}^n \binom {n} {m}  \binom {\frac{n+m-1}{2}} {n} \left[\frac{1}{m+1}\right.\right.\\
	\left.\left.+\frac{1}{(m+1)\alpha^22^{m-1}}\sum_{k=0}^{\frac{m}{2}}\binom{m+1}{k}\frac{\cos(m-2k+1)\alpha -1}{(m-2k+1)^2}\right]\right.\\
	\left.+\sum_{m\equiv 1 (\text{mod} 2)}^n \binom {n} {m}  \binom {\frac{n+m-1}{2}} {n} \left[\frac{1}{m+1}- \frac{1}{(m+1)2^{m+1}}\binom{m+1}{\frac{m+1}{2}}\right.\right.\\
	\left.\left.+\frac{1}{(m+1)\alpha^22^{m-1}}\sum_{k=0}^{\frac{m-1}{2}}\binom{m+1}{k}\left[\frac{\cos(m-2k+1)\alpha-1}{(m-2k+1)^2} \right]\right]\right\}.
	\end{multline}
\end{proposition}
\begin{proof}
	Again, a proof by direct construction is provided. We first use Equation~(\ref{Eq:4.2}) to obtain 
	\begin{align}\label{Eq:4.12}
	b_{n,2}(\alpha)=&\frac{(2n+1)}{2^{1-n}} \sum_{m=0}^n \binom {n} {m}  \binom {\frac{n+m-1}{2}} {n}\int_0^\pi \left(1-\frac{\theta}{\alpha}\right)_+^2 (\cos\theta)^m\sin\theta {\rm d} \theta\nonumber\\
	=&\frac{(2n+1)}{2^{1-n}} \sum_{m=0}^n \binom {n} {m}  \binom {\frac{n+m-1}{2}} {n}\int_0^\alpha \left(1-\frac{2\theta}{\alpha}+\frac{\theta^2}{\alpha^2}\right)(\cos\theta)^m\sin\theta {\rm d} \theta.
	\end{align}
	We now note that
	\begin{multline}\label{Eq:4.13}
	\int_0^\alpha \left(1-\frac{2\theta}{\alpha}+\frac{\theta^2}{\alpha^2}\right)(\cos\theta)^m\sin\theta {\rm d} \theta = \underbrace{\int_0^\alpha(\cos\theta)^m\sin\theta {\rm d} \theta}_{\mathcal{I}_1} \\ - \frac{2}{\alpha}\underbrace{\int_0^\alpha\theta(\cos\theta)^m\sin\theta {\rm d} \theta}_{\mathcal{I}_2}
	+\frac{2}{\alpha^2}\underbrace{\int_0^\alpha\theta^2(\cos\theta)^m\sin\theta {\rm d} \theta}_{\mathcal{I}_3}.
	\end{multline}
	The integral $\mathcal{I}_1$ is given by 
	\begin{equation}\label{Eq:4.14}
	\mathcal{I}_1=\frac{1}{m+1}-\frac{(\cos\alpha)^{m+1}}{m+1}.
	\end{equation}
	Using integration by parts, we obtain that
	\begin{equation*}
	\mathcal{I}_2= \frac{-\alpha}{m+1}(\cos\alpha)^{m+1}+\frac{1}{m+1} \int_0^\alpha(\cos\theta)^{m+1}\D\theta.
	\end{equation*}
	By using the explicit formulas 3 and 4 in \cite[Page 153]{jeffrey_2007}, we have that, when $m$ is even, 
		\begin{align}
		\mathcal{I}_2&=\frac{-\alpha}{m+1}(\cos\alpha)^{m+1} +\frac{1}{(m+1)2^m}\sum_{k=0}^{\frac{m}{2}}\binom{m+1}{k}\frac{\sin(m-2k+1)\alpha}{m-2k+1},\label{Eq:4.16}
		\end{align}
while for odd $m$ 
		\begin{multline}
		\mathcal{I}_2=
		\frac{-\alpha}{m+1}(\cos\alpha)^{m+1}+\frac{\alpha}{(m+1)2^{m+1}}\binom{m+1}{\frac{m+1}{2}} \\
	\qquad+ \frac{1}{(m+1)2^m}\sum_{k=0}^{\frac{m-1}{2}}\binom{m+1}{k}\frac{\sin(m-2k+1)\alpha}{m-2k+1}.\label{Eq:4.17}
		\end{multline}
	To compute the integral $\mathcal{I}_3$, we use integration by parts to get 
	\begin{equation}\label{Eq:4.18}
	\mathcal{I}_3=\frac{-\alpha^2}{m+1}(\cos\alpha)^{m+1}+\frac{2}{m+1}\int_0^\alpha\theta(\cos\theta)^{m+1}\D\theta.
	\end{equation}
	Using the explicit formulas 6 and 7 \cite[Page 215]{jeffrey_2007} to compute the integral in the second term of (\ref{Eq:4.18}). After then substitute in (\ref{Eq:4.18}) to obtain the integral $\mathcal{I}_3$ as follows: when $m$ is even 
		\begin{multline}
		\mathcal{I}_3 = \frac{-\alpha^2}{m+1}(\cos\alpha)^{m+1} +\frac{1}{(m+1)2^{m-1}}\sum_{k=0}^{\frac{m}{2}}\binom{m+1}{k}\binom{m+1}{k} \\
		\times\left[\frac{\alpha\sin(m-2k+1)\alpha}{m-2k+1}+\frac{\cos(m-2k+1)\alpha-1}{(m-2+k)^2} \right].\label{Eq:4.19}
		\end{multline}
and for odd $m$ 
		\begin{multline}
		\mathcal{I}_3=
		\frac{-\alpha^2}{m+1}(\cos\alpha)^{m+1}+\frac{\alpha^2}{(m+1)2^{m+1}}\binom{m+1}{\frac{m+1}{2}} \\
		+ \frac{1}{(m+1)2^{m-1}}\sum_{k=0}^{\frac{m-1}{2}}\binom{m+1}{k}\sum_{k=0}^{\frac{m-1}{2}}\binom{m+1}{k} \\
		\times \left[\frac{\alpha\sin(m-2k+1)\alpha}{m-2k+1}+\frac{\cos(m-2k+1)\alpha-1}{(m-2+k)^2} \right]\label{Eq:4.20}
		\end{multline}
	Then, from (\ref{Eq:4.14}), (\ref{Eq:4.16}), (\ref{Eq:4.17}), (\ref{Eq:4.19}) and (\ref{Eq:4.20}) in (\ref{Eq:4.13}), we have that, for even $m$, 
		\begin{align}
		\mathcal{I}_1-\frac{2\mathcal{I}_2}{\alpha}+\frac{\mathcal{I}_3}{\alpha^2}&=\frac{1}{m+1}+\frac{2^{1-m}}{(m+1)\alpha^2}\sum_{k=0}^{\frac{m}{2}}\binom{m+1}{k}\frac{\cos(m-2k+1)\alpha -1}{(m-2k+1)^2}, \label{Eq:4.21}
		\end{align}
and for odd $m$,		%
		\begin{multline}
		\mathcal{I}_1-\frac{2\mathcal{I}_2}{\alpha}+\frac{\mathcal{I}_3}{\alpha^2} =\frac{1}{m+1}+\frac{1}{(m+1)2^{m+1}}\binom{m+1}{\frac{m+1}{2}} \\
		+ \frac{2^{1-m}}{(m+1)\alpha^2}\sum_{k=0}^{\frac{m-1}{2}}\binom{m+1}{k}\left[\frac{\cos(m-2k+1)\alpha-1}{(m-2k+1)^2}\right].\label{Eq:4.22}
		\end{multline}
	Going back by substitution into Equation (\ref{Eq:4.21}) and (\ref{Eq:4.22}) in (\ref{Eq:4.12}), we obtain (\ref{Eq:4.11}). This completes the proof.
\end{proof}

\begin{figure}[hbp]
	\begin{center}
		\begin{tabular}{c}
			\includegraphics[width=0.3\textwidth,height=35mm]{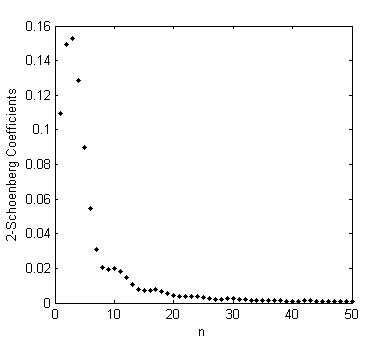}
			\includegraphics[width=0.3\textwidth,,height=35mm]{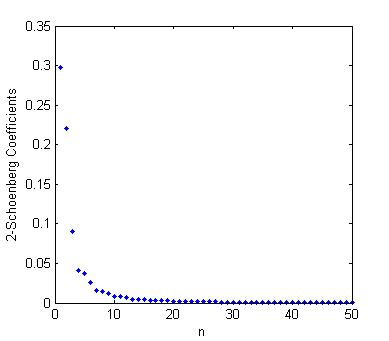}
			\includegraphics[width=0.3\textwidth,,height=35mm]{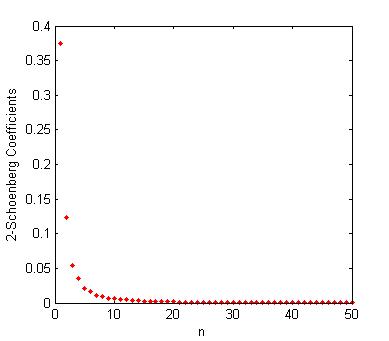} 
			\\
		\end{tabular}
		\begin{tabular}{c}
			\includegraphics[width=0.3\textwidth,height=35mm]{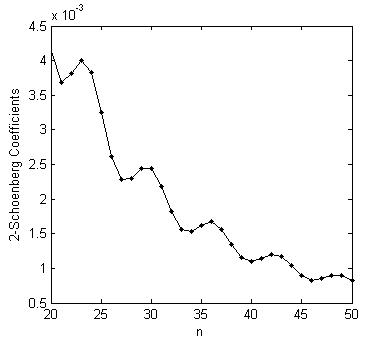}
			\includegraphics[width=0.3\textwidth,,height=35mm]{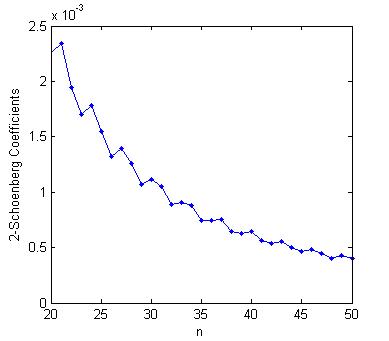}
			\includegraphics[width=0.3\textwidth,,height=35mm]{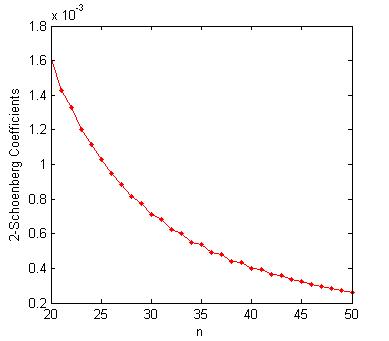} 
			\\
		\end{tabular}
	\end{center}
	\caption{
		2-Schoenberg coefficients of the Askey function $\psi_{\alpha,2}$ in Equation (\ref{Eq:4.10}), for $\alpha= 1,2,3$ (from left to right). Top: from 1st to 50-th coefficient; bottom: from 20-th to 50-th coefficient, showing the oscillating behavior away from the origin.}\label{fig:3}
\end{figure}

\section*{Acknowledgements}

We are indebted to Jochen Fiedler for valuable discussions during the preparation of the manuscript.

Funding: Ahmed Arafat and Pablo Gregori's research are supported by Spanish \emph{Ministerio de Econom\'{\i}a, Industria y Competitividad} (project MTM2016-78917-R) and \emph{Universitat Jaume I de Castell\'{o}n} (project P1$\cdot$1B2015-40). Emilio Porcu is supported by Proyecto Fondecyt number 1170290.

\bibliographystyle{plain}
\bibliography{bibliography}

\end{document}